\documentclass[a4paper,10pt]{elsarticle}

\usepackage[utf8x]{inputenc}
\usepackage{amsmath}
\usepackage{amsfonts}
\usepackage{amssymb}
\usepackage{graphicx}
\usepackage{subcaption}
\usepackage{hyperref}
\usepackage{color}
\usepackage{mathtools}
\usepackage[noabbrev,capitalize]{cleveref}
\begin{document}
 \title{A note on 2-vertex-connected orientations}
\author{Florian H\"orsch}
 \ead{florian.hoersch@tu-ilmenau.de}
\address{TU Ilmenau, Weimarer Straße 25, Ilmenau, Germany, 98693.  }
\author{Zolt\'an Szigeti}
\ead{ Zoltan.Szigeti@grenoble-inp.fr}
\address{Univ.~Grenoble~Alpes, Grenoble INP, CNRS, G-SCOP, 46 Avenue F\'elix Viallet, Grenoble, France, 38000.  }
\date{\today}

\newtheorem{Statement}{Statement}
\newtheorem{rem}{Remark}
\newtheorem{Proposition}{Proposition}
\newtheorem{coro}{Corollary}
\newtheorem{Theorem}{Theorem}
\newtheorem{Conjecture}{Conjecture}
\newtheorem{Claim}{Claim}
\newtheorem{Lemma}{Lemma}
\newtheorem{prob}{Problem}
\newtheorem{case}{Case}

\newenvironment{proof}{\noindent \textbf{Proof}}{\rule{2mm}{2mm}}

\begin{abstract}
We consider two possible extensions of a theorem of Thomassen characterizing the graphs admitting a 2-vertex-connected orientation. First, we show that the problem of deciding whether a mixed graph has a 2-vertex-connected orientation is NP-hard. This answers a question of Bang-Jensen, Huang and Zhu. For the second part, we call a directed graph $D=(V,A)$ $2T$-connected for some $T \subseteq V$ if $D$ is 2-arc-connected and $D-v$ is strongly connected for all $v  \in T$. We deduce a characterization of the graphs admitting a $2T$-connected orientation from the theorem of Thomassen.
\end{abstract}
\maketitle

\section{Introduction}\label{intro}
In this article, we deal with two possible extensions of a theorem of Thomassen characterizing graphs having a 2-vertex-connected orientation. All undefined notions can be found in Section \ref{prel}.

During the history of graph orientations, the question of characterizing graphs having orientations with certain connectivity properties has played a central role. The following fundamental theorem of Robbins \cite{robb} dates back to 1939.

\begin{Theorem}
A graph has a strongly connected orientation if and only if it is 2-edge-connected.
\end{Theorem}

For higher arc-connectivity, this theorem was later generalized by Nash-Williams \cite{N60}.

\begin{Theorem}\label{nw}
Let $G$ be a graph and $k$ a positive integer. Then $G$ has a $k$-arc-connected orientation if and only if $G$ is $2k$-edge-connected.
\end{Theorem}

The analogous problem for vertex-connectivity turns out to be much more complicated. The following conjecture was proposed by Frank in \cite{F-HC95}.

\begin{Conjecture}\label{fra}
Let $G=(V,E)$ be a graph and $k$ a positive integer. Then $G$ has a k-vertex-connected orientation if and only if $|V|\geq k+1$ and $G-X$ is $2(k-|X|)$-edge-connected for all $X \subseteq V$ with $|X|\leq k-1$.
\end{Conjecture}

Although Conjecture \ref{fra} remained open for a long time, little progress was made on it. Finally, Conjecture \ref{fra} was proven for $k=2$ by Thomassen \cite{CT}. More explicitly, he proved the following  theorem.

\begin{Theorem}\label{thom2}
A graph $G$ has a 2-vertex-connected orientation if and only if $G$ is 4-edge-connected and $G-v$ is 2-edge-connected for all $v \in V$.
\end{Theorem}

On the other hand, Conjecture \ref{fra} was disproven for every $k \geq 3$ by Durand de Gevigney \cite{ODG}. Moreover, he proved the following result which makes a good characterization of the graphs admitting a $k$-vertex-connected orientation for any $k \geq 3$ seem out of reach.

\begin{Theorem}\label{olivier}
The problem of deciding whether a given graph has a $k$-vertex-connected orientation is NP-hard for any $k \geq 3$.
\end{Theorem}

It remains interesting to search for some big class of graphs that admit highly vertex-connected orientations. The following conjecture was proposed by Thomassen \cite{CT2}.

\begin{Conjecture}\label{appro}
There is a function $f:\mathbb{Z}_+\rightarrow \mathbb{Z}_+$ such that every $f(k)$-vertex-connected graph has a $k$-vertex-connected orientation for all $k \in \mathbb{Z}_+$.
\end{Conjecture}

Conjecture \ref{appro} remains open for all $k \geq 3$.
\medskip

In this article, we deal with two possible extensions of Theorem \ref{thom2}.
In the first part, we deal with a possible generalization of Theorem \ref{thom2} to the case when some of the edges are pre-oriented. The following is the first important result on orientations of mixed graphs satisfying connecitivity properties. It was proven by Boesch and Tindell \cite{bt}.

\begin{Theorem}
A mixed graph $G=(V, A \cup E)$ has a strongly connected orientation if and only if $d_A^-(X)+\frac{1}{2}d_E(X)\geq 1$ for every nonempty $X \subsetneq V$.
\end{Theorem}

For general arc-connectivity, this problem has been solved by Frank \cite{book} who obtained a pretty technical characterization of mixed graphs admitting a $k$-arc-connected orientation for all $k \in \mathbb{Z}_+$ using the theory of generalized polymatroids.
For higher vertex-connectivity, the possibility of a good characterization of the mixed graphs admitting a $k$-vertex-connected orientation has been ruled out by Theorem \ref{olivier} for any $k \geq 3$. However, the case of $k=2$ remained open. 
\medskip

The first main contribution of this work is to show that there is also no hope to find a good characterization for this problem. More formally, we consider the following algorithmic problem:
\medskip

\noindent \textbf{2-vertex-connected orientation of mixed graphs (2VCOMG):}
\smallskip

\noindent\textbf{Input:} A mixed graph $G=(V,A \cup E)$.
\smallskip

\noindent\textbf{Question:} Does $G$ have a 2-vertex-connected orientation?
\medskip

The question of determining the complexity of this problem was first hinted at by Thomassen in \cite{CT} and then asked explicitely by Bang-Jensen, Huang and Zhu \cite{bhz}. Our main contribution is the following answer to this problem.

\begin{Theorem}\label{thom}
2VCOMG is NP-hard.
\end{Theorem}

Our reduction that proves Theorem \ref{thom} is inspired by the one used by Durand de Gevigney when proving Theorem \ref{olivier}.
\medskip

In the second part, we deal with a connectivity property that generalizes both 2-vertex-connectivity and 2-arc-connectivity and was introduced by Durand de Gevigney and the second author in \cite{DS}. Namely a given digraph $D=(V,A)$ is called {\it $2T$-connected }for some $T \subseteq V$ if $D$ is 2-arc-connected and $D-v$ is strongly connected for all $v \in T$. We prove the following theorem characterizing the graphs $G=(V,E)$ admitting a $2T$-connected orientation for some given $T \subseteq V$.

\begin{Theorem}\label{huoh}
Let $G$ be a graph and $T \subseteq V(G)$. Then $G$ has a $2T$-connected orientation if and only if $G$ is $4$-edge-connected and $G-v$ is 2-edge-connected for all $v \in T$.
\end{Theorem}

Observe that Theorem \ref{huoh} implies both Theorem \ref{thom2} and Theorem \ref{nw} for $k=2$ as $2T$-connectivity corresponds to 2-arc-connectivity for $T=\emptyset$ and to 2-vertex-connectivity for $T=V$. The proof of Theorem \ref{huoh} works by a rather simple deduction from Theorem \ref{thom2}. It would be nice to find a proof of Theorem \ref{huoh} that does not use Theorem \ref{thom2} and hence to get a transparent proof of Theorem \ref{thom2}.
\medskip

The rest of this article is structured as follows: In Section \ref{prel}, we give some more formal definitions and some preliminary results. In Section \ref{redu}, we give the reduction that proves Theorem \ref{thom}. In Section \ref{2-T}, we prove Theorem \ref{huoh}. Finally, in Section \ref{conc}, we conclude our work.
\section{Preliminaries}\label{prel}

We first give some basic notation in graph theory. A {\it mixed graph} consists of a vertex set $V$,  an arc set $A$ and an edge set $E$. If $A=\emptyset$, then $G$ is a graph and if $E=\emptyset$, then $G$ is a digraph. For a single vertex $v$, we often use {\boldmath$v$} instead of $\{v\}$. For some mixed graph $G=(V,A \cup E)$ and some $X \subseteq V$, we use {\boldmath$d_A^-(X)$} for the number of arcs in $A$ whose tail is in $V-X$ and whose head is in $X$, {\boldmath$d_A^+(X)$} for $d_A^-(V-X)$ and {\boldmath$d_E(X)$} to denote the number of edges in $E$ that have exactly one endvertex in $X$. For some $u,v \in V$, an {\it $uv$-path} in $G$ is a sequence of vertices $v_1,\ldots,v_t$ sucht that $u=v_0,v=v_t$ and for all $i=0,\ldots,t-1$ either $v_iv_{i+1}\in E$ or $v_iv_{i+1}\in A$. Two $uv$-paths are called {\it internally disjoint} if they share no vertices apart from $u$ and $v$. 
For a vertex set $X\subseteq V$ and a vertex $v \in V-X$, a {\it $(v,X)$-path} is a path from $v$ to a vertex of $X$. Similarly, a {\it $(X,v)$-path} is a path from a vertex of $X$ to $v$.
Further, for some $X \subseteq V$, $G$ is called {\it $k$-vertex-connected in $X$} if $|V|\geq k+1$ and there are $k$ internally disjoint $uv$-paths for any $u,v \in X$. Also, $G$ is called {\it $k$-vertex-connected} if $G$ is $k$-vertex-connected in $V$. 
For some $X \subseteq V$, we denote by {\boldmath$G[X]$} the subgraph of $G$ induced on $X$.

A graph $G=(V,E)$ is called {\it $k$-edge-connected}  for some positive integer $k$ if $d_E(X)\geq k$ for every nonempty $X \subsetneq V$. A digraph $D=(V,A)$ is called {\it $k$-arc-connected} for some positive integer $k$ if $d_A^-(X)\geq k$ for every nonempty $X \subsetneq V$. If $D$ is $1$-arc-connected then we say that it is {\it strongly connected}.

A connected graph with every vertex of degree $2$ is called a {\it cycle} and  a {\it double cycle} is obtained from a cycle by duplicating every edge. A strongly connected orientation of a cycle is called a {\it circuit}. A digraph $D=(V,A)$ whose underlying graph does not contain a cycle is called an {\it $r$-in-arborescence ($r$-out-arborescence)} if $r\in V$ and $D$ contains a path from $v$ to $r$ (from $r$ to $v$) for every  $v\in V$.

Given two graphs $G$ and $H$ and a vertex $v$ of $G$, {\it blowing up} $v$ into $H$ means that we replace $v$ by $H$ and we replace every edge $wv$ incident to $v$ in $G$ by an edge $wu$ for some vertex $u$ in $H.$
\medskip

We now give one basic result on vertex-connectivity in digraphs.

\begin{Proposition}\label{vert}
Let $D=(V,A)$ be a digraph, $X \subseteq V$ such that $D$ is 2-vertex-connected in $X$ and $v \in V-X$. If $D$ contains two $(v,X)$-paths whose vertex sets only intersect in $v$ and $D$ contains two $(X,v)$-paths whose vertex sets only intersect in $v$, then $D$ is 2-vertex-connected in $X \cup v$.
\end{Proposition}

We also need  one property on edge-connectivity in graphs.

\begin{Proposition}\label{blowup}
Given two graphs $G$ and $H$  and a vertex $v$ of $G,$ if $G$ and $H$ are $k$-edge-connected then so is the graph obtained from $G$ by blowing up $v$ into $H$.
\end{Proposition}

The algorithmic problem we need for our reduction is MNAE3SAT.
\medskip

\noindent {\bf Monotone not-all-equal-3SAT (MNAE3SAT)}
\smallskip

\noindent {\bf Input:} A set $X$ of boolean variables, a formula consisting of a set $\mathcal{C}$ of clauses each containing 3 distinct variables, none of which are negated.
\smallskip

\noindent {\bf Question:} Is there a truth assignment to the variables of $X$ such that every clause in $\mathcal{C}$ contains at least one true and at least one false literal?

An assignment satisfying the above condition will be called {\it feasible}.
\medskip

This problem will be used in the reduction which is justified by the following result due to Schaefer \cite{schaefer}.

\begin{Theorem}\label{sathard}
MNAE3SAT is NP-complete.
\end{Theorem}

\section{The reduction}\label{redu}

Let $\Phi=(X,\mathcal{C})$ be an instance of MNAE3SAT.  The set of pairs $(x,C)$ such that $x\in C\in\mathcal{C}$ is denoted by $P(\Phi).$ In the following, we first create an instance of 2VCOMG and then show that it is a positive instance if and only if $\Phi$ is a positive instance of MNAE3SAT.
\medskip

We construct a mixed graph $G=(V,A \cup E)$ as follows. First, let $V$ contain a set $Q$ of three vertices $p,q$ and $r$. Further, $V$ contains a set $Z$ containing one vertex $z_C$ for every $C \in \mathcal{C}$. Finally, for every $(x,C)\in P(\Phi)$, $V$ contains a set $R_C^x$ of $4$ vertices $\{t_C^x,u_C^x,w_C^x,y_C^x\}$.  First, let $A$ contain the arcs $pq,qp,pr,rp,qr,rq$. Further, for every $C \in \mathcal{C}$, $A$ contains  the arcs  $pz_C$ and $z_Cq$. Finally, for every $(x,C)\in P(\Phi)$, $A$ contains the arcs of the path $p,t_C^x,u_C^x,y_C^x,u_C^x, w_C^x,q$. First, let $E$ contain an edge $z_Cu_C^x$ for every $(x,C)\in P(\Phi)$. Now for every $x \in X$, let $C_1,\ldots,C_{\mu (x)}$ be an arbitrary ordering of the clauses in $\mathcal{C}$ containing $x$. Let $b_1^x=r$ and for $i=1,\ldots ,\mu(x)$,  let $b_{3i-1}^x=y_{C_i}^x, b_{3i}^x=w_{C_i}^x$ and $b_{3i+1}^x=t_{C_i}^x$. We add the edges of the cycle $B^x=b_1^x,b^x_2,\ldots,b^x_{3\mu(x)+1},b^x_1$ to $E$. This finishes the construction of $G$. Note that the size of $G$ is clearly polynomial in the size of $\Phi$. A drawing can be found in Figure \ref{constr}. 

\begin{figure}[h]
\begin{center}
  \includegraphics[width=.23 \linewidth]{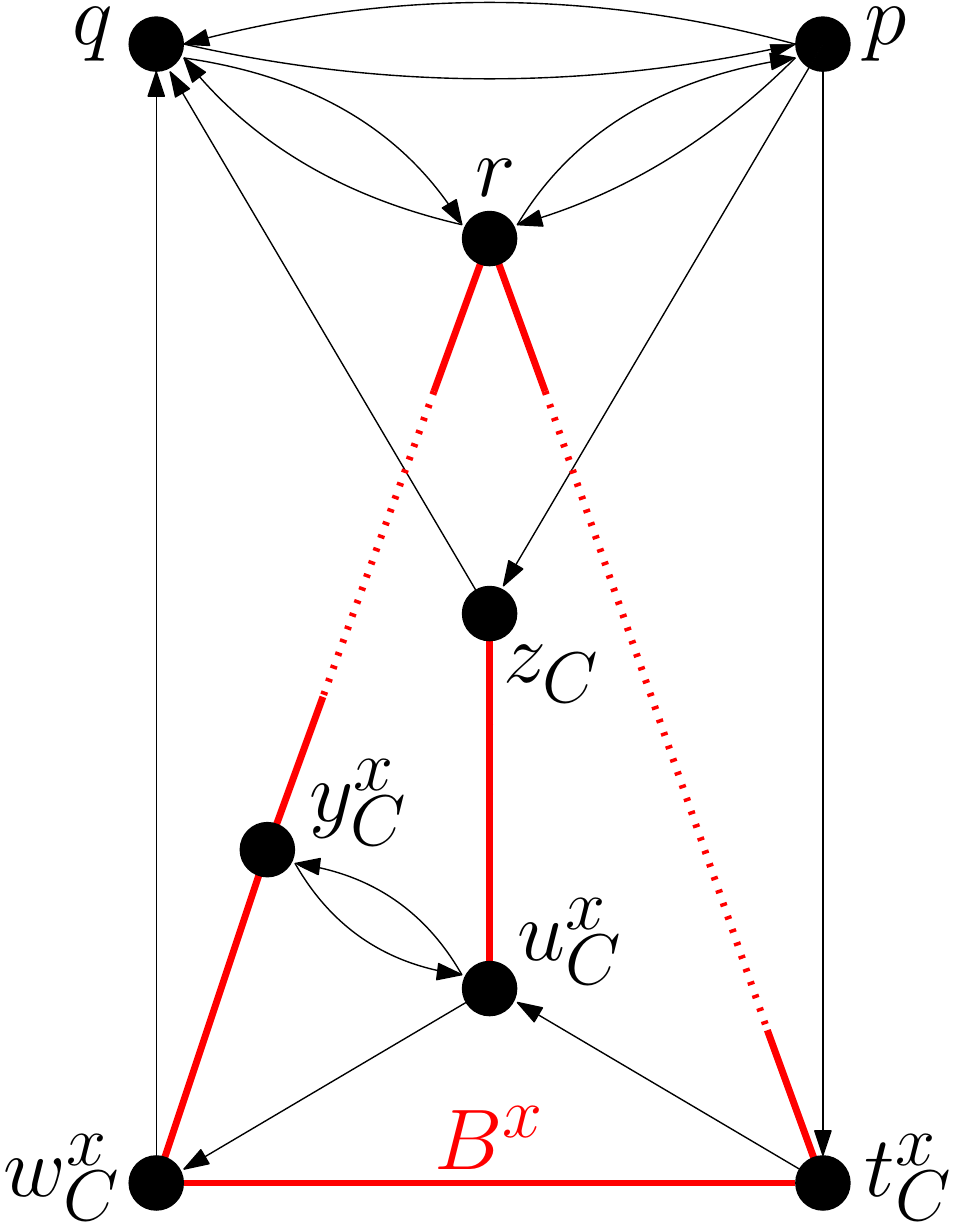}
  \caption{A schematic drawing of $G$ containing $Q$ and $R_C^x$ and $z_C$ for some $(x,C) \in P(\Phi)$.}\label{constr}
\end{center}
\end{figure}

For some $x \in X,$ we will refer to the circuit $b_1^x,b_2^x,\ldots,b_{3\mu(x)+1}^x,b_1^x$ as $\overrightarrow{B^x}$  and to the circuit $b_1^x,b_{3\mu(x)+1}^x,\ldots,b_2^x,b_1^x$ as $\overleftarrow{B^x}$.
\medskip

To  show that $G$ is a positive instance of 2VCOMG if and only if $\Phi$ is a positive instance of MNAE3SAT we need the following lemma.

\begin{Lemma}\label{rdftgzuuh}
An orientation $\vec{G}=(V,A \cup \vec{E})$ of $G$ is 2-vertex-connected if and only if 
\begin{eqnarray}
&&\text{$\vec{G}[B^x]=\overrightarrow{B^x}$ or $\vec{G}[B^x]=\overleftarrow{B^x}$ \hskip 1.8truecm for every $x \in X$,}\label{a}\\
&&\text{$u_C^xz_C\in \vec{E}$ if and only if $\vec{G}[B^x]=\overrightarrow{B^x}$ \hskip .4truecm for every $(x,C)\in P(\Phi)$,}\label{b}\\
&&\text{$u_C^{x_1}z_C, z_Cu_C^{x_2}\in\vec{E}$ for some $x_1,x_2\in C$ \hskip .2truecm for every $C \in \mathcal{C}$.}\label{c}
\end{eqnarray}
\end{Lemma}

\begin{proof}
First suppose that $\vec{G}$  is 2-vertex-connected. 

Since for every $(x,C)\in P(\Phi)$, the vertices $t_C^x,w_C^x$ and $y_C^x$ have one arc entering in $A$, one arc leaving in $A$ and two edges entering in $E$, \eqref{a} follows.

Let $(x,C)\in P(\Phi)$. For some $i$, we have $y_C^x=b_{3i-1}^x.$ Since $\vec G-t_C^x$ is strongly connected, $\{u_C^x,w_C^x,y_C^x\}$ has no arc entering in $A$ and two edges entering in $E$, at least one of $z_Cu_C^x$ and $b_{3i-2}^xb_{3i-1}^x$ exists in $\vec{E}$. Since $\vec G-w_C^x$ is strongly connected, $\{u_C^x,y_C^x\}$ has no arc leaving in $A$ and two edges entering in $E$, at least one of $u_C^xz_C$ and $b_{3i-1}^xb_{3i-2}^x$ exists in $\vec{E}$. We obtain that $u_C^xz_C\in \vec{E}$ if and only if $b_{3i-2}^xb_{3i-1}^x \in \vec{E}$. Now \eqref{a} yields  \eqref{b}.

Since for every $C \in \mathcal{C}$, the vertex $z_C$ has one arc entering in $A$, one arc leaving in $A$ and three edges entering in $E$, \eqref{c} follows.
\medskip

Now suppose that \eqref{a}, \eqref{b} and \eqref{c} hold.

We first show that $\vec{G}$  is 2-vertex-connected  in $Q\cup R_C^x$ for every $(x,C)\in P(\Phi)$. We fix some  $(x,C)\in P(\Phi)$ and for convenience,  we denote $z_C,t_C^x,u_C^x,w_C^x,y_C^x$ by $z,t,u,w,y$, respectively.  Note that $\vec{G}[Q]$ is 2-vertex-connected. We distinguish two cases depending on the orientation of $B^x$ in $\vec{G}$. By \eqref{a}, we have either $\vec{G}[B^x]=\overrightarrow{B^x}$ or $\vec{G}[B^x]=\overleftarrow{B^x}$.
\medskip

\noindent {\bf Case 1.} $\vec{G}[B^x]=\overrightarrow{B^x}$. Observe that  $\vec{G}[B_x]$ consists of a path $S_1$ from $r$ to $y$ disjoint from  $\{t, w\}$, of the arcs $yw, wt$ and of a path $S_2$ from $t$ to $r$  disjoint from $\{y, w\}$. By \eqref{b}, we have  $uz\in \vec{E}$. Let $F_1$ be the $r$-out-arborescence consisting of $S_1$ and the  arcs $yu, yw$ and $wt$. Let $F_2$ be the $p$-out-arborescence consisting of  the  arcs $pt, tu, uw$ and $uy$. Then $F_1$ and $F_2$ contain two $(Q,v)$-paths whose vertex sets only intersect in $v$  for every vertex $v$ in $R_C^x.$  Let $F_3$ be the $r$-in-arborescence consisting of $S_2$ and the  arcs $yw, uw$ and $wt$. Let $F_4$ be the $q$-in-arborescence consisting of  the  arcs $tu, yu, uz, zq$ and $wq$. Then $F_3$ and $F_4$ contain two $(v,Q)$-paths whose vertex sets only intersect in $v$  for every vertex $v$ in $R_C^x$.  An illustration can be found in Figure \ref{fig:case1}.

\begin{figure}[h]
  \includegraphics[width=.45 \linewidth]{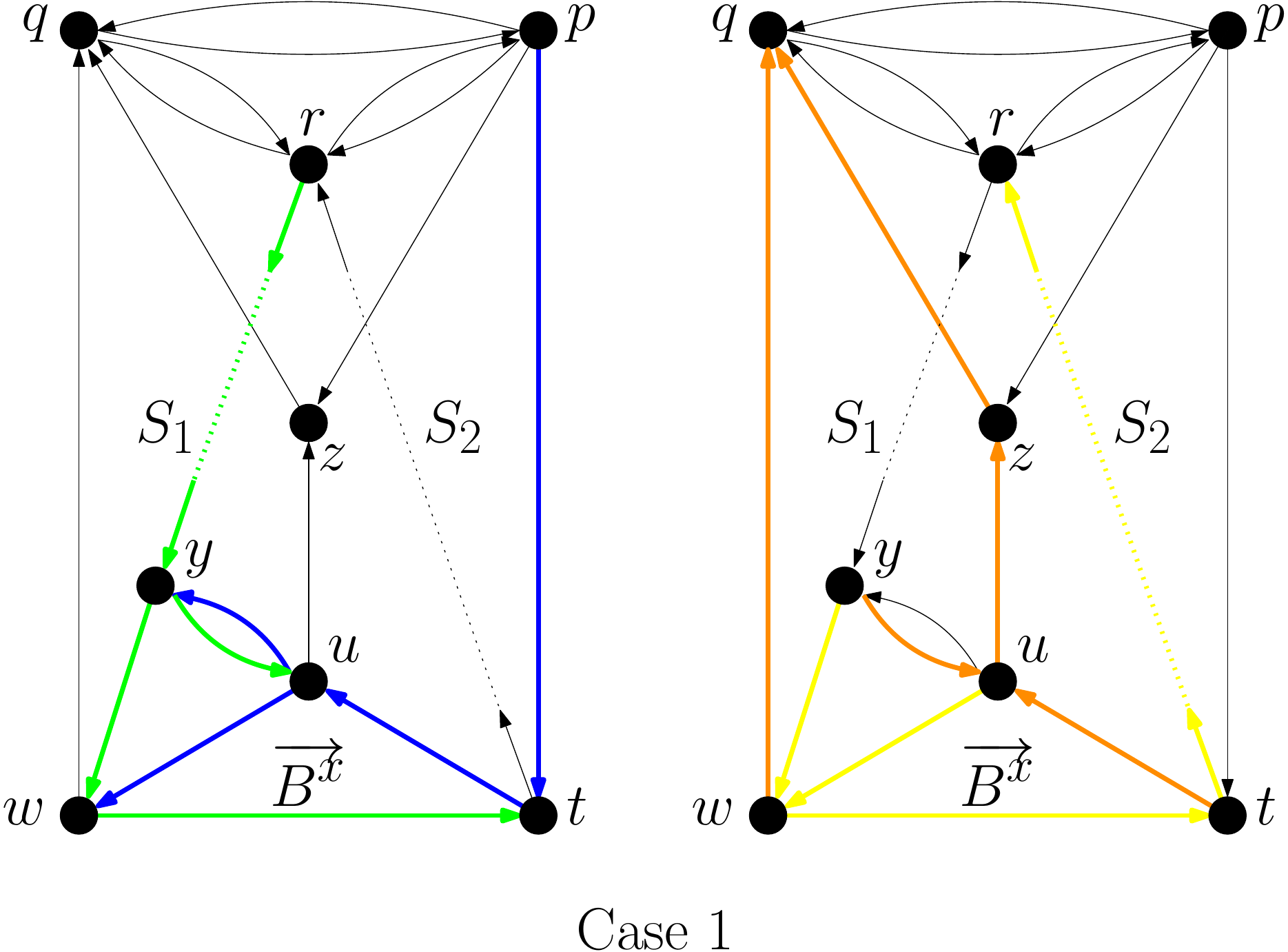}  \hskip 1truecm \includegraphics[width=.45 \linewidth]{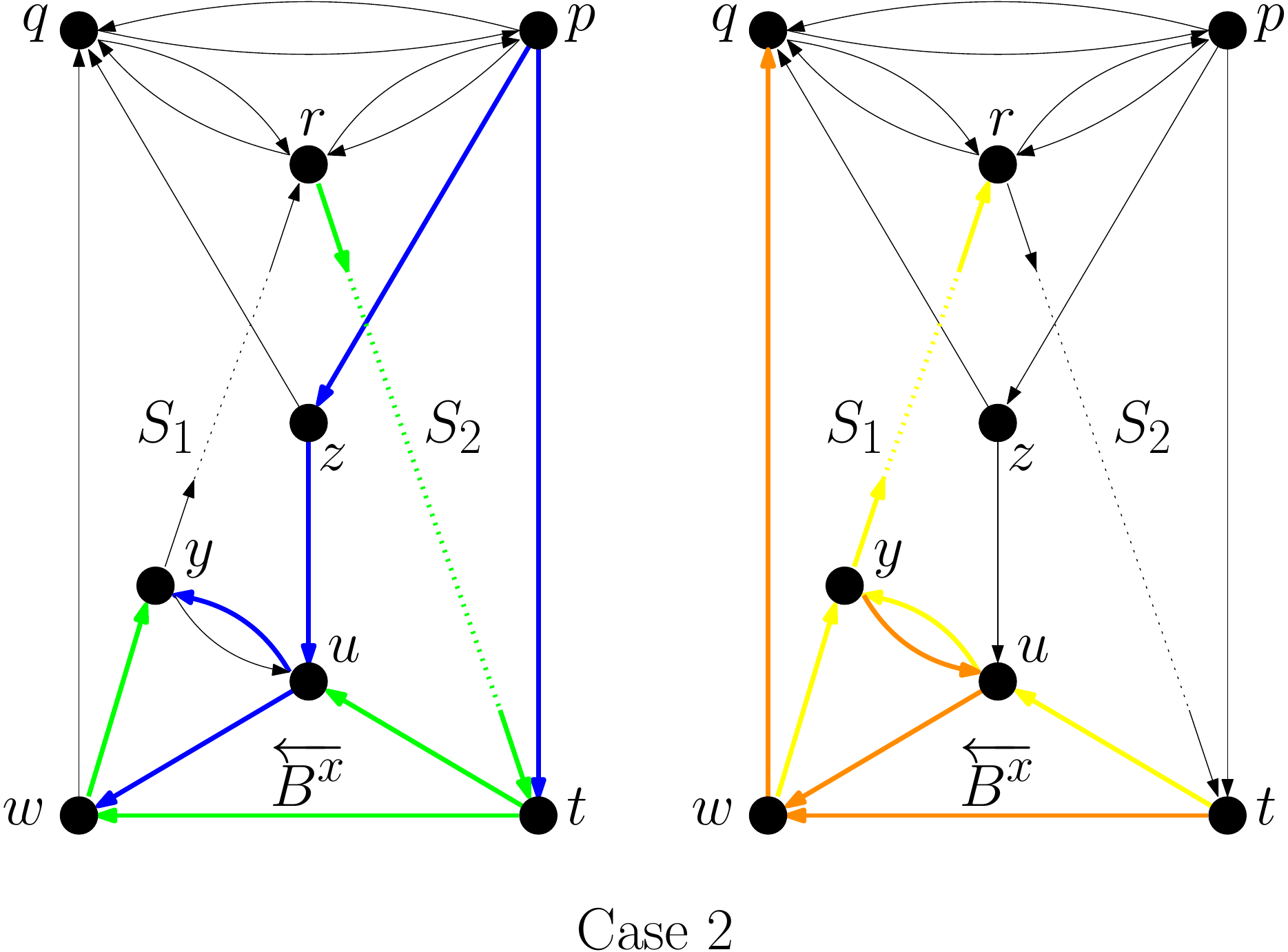}
  \caption{An illustration for the two cases in the proof of Lemma \ref{rdftgzuuh}. The out-arborescences $F_1,F_2$ and  the in-arborescences $F_3,F_4$ are depicted in green, blue, yellow and orange, respectively.}\label{fig:case1}
\end{figure}
 \medskip

\noindent {\bf Case 2.} $\vec{G}[B^x]=\overleftarrow{B^x}$. Observe that $\vec{G}[B_x]$ consists  of a path $S_1$ from $y$ to $r$ disjoint from $\{t, w\}$,  of the arcs $tw, wy$ and  of a path $S_2$ from  $r$ to $t$ disjoint from $\{y, w\}$. 
By \eqref{b}, we have $zu\in \vec{E}$. Let $F_1$ be the $r$-out-arborescence consisting of $S_2$ and the  arcs $tu, tw$ and $wy$. Let $F_2$ be the $p$-out-arborescence consisting of  the  arcs $pt, pz, zu, uw$ and $uy$. Then $F_1$ and $F_2$ contain two $(Q,v)$-paths whose vertex sets only intersect in $v$  for every vertex $v$ in $R_C^x.$ Let $F_3$ be the $r$-in-arborescence consisting of $S_1$ and the  arcs $tu, uy$ and $wy$. Let $F_4$ be the $q$-in-arborescence consisting of  the  arcs $tw, uw,yu$ and $wq$. Then $F_3$ and $F_4$ contain two $(v,Q)$-paths whose vertex sets only intersect in $v$  for every vertex $v$ in $R_C^x$.  An illustration can be found in Figure \ref{fig:case1}.
\medskip

In either case, we obtain by Proposition \ref{vert}, that $\vec{G}$ is 2-vertex-connected in $Q\cup R_C^x$. 
As $(x,C)$ was chosen arbitrarily, we in fact obtained that $\vec{G}$ is 2-vertex-connected in $V-Z.$
\medskip

To finish the proof we  consider some $C\in\mathcal{C}$. By \eqref{c},  $u_C^{x_1}z_C, z_Cu_C^{x_2}\in\vec{E}$ for some $x_1,x_2\in C$. Further,  $z_Cq,pz_C\in A$. Then Proposition \ref{vert} yields that $\vec{G}$  is 2-vertex-connected in $(V-Z)\cup z_C$. As $C$ was chosen arbitrarily, the proof of Lemma \ref{rdftgzuuh} is finished.
\end{proof}

\begin{Lemma}\label{dbjkeb}
There exists a feasible truth assignment for $\Phi$ if and only if $G$ has a $2$-vertex-connected orientation.
\end{Lemma}

\begin{proof}
First suppose that there exists a feasible truth assignment $f:X \rightarrow \{true,false\}$ for $\Phi$. We  create an orientation $\vec{G}$ of $G$ in the following way: for every $x \in X$, we orient $B^x$ as $\overrightarrow{B^x}$ if $f(x)=true$ and as $\overleftarrow{B^x}$ if $f(x)=false$. Further, for every $(x,C)\in P(\Phi)$, we orient $z_Cu_C^x\in E$ from $u_C^x$ to $z_C$ if  $f(x)=true$ and  from $z_C$ to $u_C^x$ if $f(x)=false$.  Observe that \eqref{a} and \eqref{b} hold. Since $f$ is feasible  for $\Phi$,  \eqref{c} also holds. Then, by Lemma \ref{rdftgzuuh}, $\vec{G}$ is $2$-vertex-connected.
\medskip

Now suppose that $G$ has a $2$-vertex-connected orientation $\vec{G}$. Then, by Lemma \ref{rdftgzuuh}, \eqref{a}, \eqref{b} and \eqref{c}  hold. For every $x \in X$, by \eqref{a}, we have $\vec{G}[B^x]=\overrightarrow{B^x}$ or $\vec{G}[B^x]=\overleftarrow{B^x}$. We can hence define a truth assignment $f$ as follows: we set $f(x)=true$ if $\vec{G}[B^x]=\overrightarrow{B^x}$ and $false$ if $\vec{G}[B^x]=\overleftarrow{B^x}$. For every $C \in \mathcal{C}$, by \eqref{c}, there exist arcs $u_C^{x_1}z_C$ and $z_Cu_C^{x_2}$ for some $x_1,x_2\in C$. By \eqref{b}, we have $\vec{G}[B^{x_1}]=\overrightarrow{B^{x_1}}$ and $\vec{G}[B^{x_2}]=\overleftarrow B^{x_2}$. We obtain that $f(x_1)=true$ and $f(x_2)=false$. This implies that $f$ is feasible for $\Phi$.
\end{proof}
\bigskip

By Lemma \ref{dbjkeb} and Theorem \ref{sathard}, the proof of Theorem \ref{thom} is finished.  

\section{Orientations for $2T$-connectivity}\label{2-T}

This section is dedicated to proving Theorem \ref{huoh}.
\medskip

\begin{proof}(of Theorem \ref{huoh})
Necessity is evident.

To prove the sufficiency, let $H$ be obtained from $G=(V,E)$ by blowing up every vertex $v \in V-T$ into a double cycle $C_v$ on a vertex set of size $\max\{3,\lceil\frac{d_G(v)}{2}\rceil\}$ such that every new vertex is incident to a set $F_v$ of at most $2$ edges not belonging to $C_v$.

\begin{Claim}\label{rfg7zuho}
$H$ is 4-edge-connected and $H-w$ is 2-edge-connected for all $w \in V(H)$.
\end{Claim}

\begin{proof}
Since $G$ and $C_v$ for all $v\in V-T$ are $4$-edge-connected, so is $H$ by Proposition \ref{blowup}.

Now let $w \in V(H)$. If $w\in T$, then since $G-w$ and $C_v$ for all $v\in V-T$ are $2$-edge-connected, so is $H-w$ by Proposition \ref{blowup}. Otherwise, $w\in V(C_u)$ for some $u\in V-T.$ Note that $G'=G-F_u$ is $2$-edge-connected because $G$ is $4$-edge-connected. Further, $C_u-u$ is $2$-edge-connected. Observe that $H-w$ is the graph obtained from $G'$ by blowing up every vertex $v \in (V-u)-T$ into $C_v$  and then blowing up $u$ into $C_u-u$. It follows, by Proposition \ref{blowup}, that $H-w$ is $2$-edge-connected.
\end{proof}

\medskip

By Claim \ref{rfg7zuho} and Theorem \ref{thom}, we obtain that $H$ has a 2-vertex-connected orientation $\vec{H}$. Now let $\vec{G}$ be obtained from contracting $V(C_v)$ into $v$ for all $v \in V(G)-T$. We will show that $\vec{G}$ is $2T$-connected. Since $\vec{H}$ is 2-vertex-connected, we obtain that $\vec{H}$ is also 2-arc-connected. As $\vec{G}$ is obtained from $\vec{H}$ through contractions, we obtain that $\vec{G}$ is also 2-arc-connected. Now let $v \in T$. Since $\vec{H}$ is 2-vertex-connected, we obtain that $\vec{H}-v$ is strongly connected. As $\vec{G}-v$ is obtained from $\vec{H}-v$ through contractions, we obtain that $\vec{G}-v$ is also strongly connected.
\end{proof}
\section{Conclusion}\label{conc}
We show that the problem of deciding whether a mixed graph has a 2-vertex-connected orientation is NP-hard and give a characterization for the graphs admitting a $2T$-connected orientation. The first result closes the dichotomy for the problem of finding $k$-vertex-connected orientations of mixed graphs.
\medskip

In the spirit of Conjecture \ref{appro}, we pose the following problem.

\begin{Conjecture}\label{neu}
There is a function $f:\mathbb{Z}_+\rightarrow \mathbb{Z}_+$ such that every $f(k)$-vertex-connected mixed graph has a $k$-vertex-connected orientation for all $k \in \mathbb{Z}_+$.
\end{Conjecture}

Clearly, for any fixed $k \geq 3$, Conjecture \ref{neu} implies Conjecture \ref{appro}. It would be interesting to see whether Conjecture \ref{neu} is tractable more easily for $k=2$.

\end{document}